\documentclass[11pt, twoside]{article}
\usepackage{amsfonts,amssymb,amsmath,amsthm}
\usepackage{graphicx}
\usepackage[all]{xy}

\usepackage{pstricks,psfrag,xmpmulti,amscd, import}
\newtheorem*{mainthm}{Main Theorem}

 \divide\oddsidemargin by 2
\newtheorem{thm}{Theorem}[section]

\newtheorem{cor}[thm]{Corollary}
\newtheorem{lem}[thm]{Lemma}
\newtheorem{rem}[thm]{Remark}

\newtheorem{prop}[thm]{Proposition}

\theoremstyle{definition}

\theoremstyle{remark}

\numberwithin{equation}{section}

\font\nt=cmr7

\def\note#1
{\marginpar
{\nt $\leftarrow$
\par
\hfuzz=20pt \hbadness=9000 \hyphenpenalty=-100 \exhyphenpenalty=-100
\pretolerance=-1 \tolerance=9999 \doublehyphendemerits=-100000
\finalhyphendemerits=-100000 \baselineskip=6pt
#1}\hfuzz=1pt}

\def\be{\begin{equation}}
\def\ee{\end{equation}}

\renewcommand{\epsilon}{\varepsilon}

\newcommand{\ra}{\rightarrow}

\newcommand{\dist}{\operatorname{dist}}

\newcommand{\TT}{{\cal T}}

\newcommand{\C}{{\mathbb C}}

\newcommand{\D}{{\mathbb D}}
\newcommand{\Hyp}{{\mathbb H}}

\newcommand{\N}{{\mathbb N}}

\newcommand{\Z}{{\mathbb Z}}

\def\B0{{\mathbf{0}}}

\newcommand{\ov}{\overline}

\renewcommand{\ra}{\rightarrow}
\renewcommand{\Re}{\text{Re }}
\newcommand{\Chat}{\hat{\C}}

\catcode`\@=12

\def\Empty{}
\newcommand\oplabel[1]{
  \def\OpArg{#1} \ifx \OpArg\Empty {} \else
  	\label{#1}
  \fi}
		
\renewcommand{\hat}{\widehat}

\renewcommand{\phi}{\varphi}

\renewcommand{\hat}{\widehat}

 \title{A note on repelling periodic points for meromorphic functions with a bounded set of  singular values}
 
\vspace{5cm}
 \author{\small Anna Miriam Benini \\
\small Universit\`a di Roma Tor Vergata \\
\small v. ricerca scientifica 1 \\   
\small Roma, Italy\\ 
\small {\tt ambenini$@$gmail.com} 
}  

\begin{document}

\maketitle  
\begin{abstract}Let $f$ be a meromorphic  function with a  bounded set of  singular values and for which infinity is a logarithmic singularity. Then we show that  $f$ has infinitely many repelling periodic points for any minimal period $n\geq1$, using a much simpler argument than the corresponding results for arbitrary entire transcendental functions.\begin{footnote}{MSC classes:	37F10, 37F20. Keywords: holomorphic dynamics, meromorphic functions, transcendental functions, repelling periodic points.}
\end{footnote}
\end{abstract}

\section{Introduction}
 
An entire transcendental function $f$ is a function which is holomorphic on all of the complex plane $\C$ and
for which infinity is an essential singularity. A meromorphic function is the quotient of two entire (not necessarily transcendental) functions and  can have \emph{poles}, that is points whose image equals to infinity and whose orbits are no further defined. 
The  \emph{set of singular values} $S(f)$ is the set of values near  which not all branches of the inverse are well defined and univalent. $S(f)$ includes critical values, asymptotic values and any of their accumulation points.  Then the function $f:\C\setminus{f^{-1}(S(f))}\ra \C\setminus S(f)$ is an unbranched covering. In general we have that $S(f^n)=S(f)\cup f(S(f))\cup\ldots\cup f^n(S)$, where  $f^n:=f\circ\ldots\circ f$ is defined to be the composition of $f$ with itself $n$ times.

A \emph{repelling periodic point} of period $n$ for $f$ is a point such that $f^n(z)=z$ and $|(f^n)'(z)|>1$. The period $n$ is called \emph{minimal} if there is no $j<n$ such that  $f^j(z)=z$. 
The problem of the existence of  periodic points of any minimal period for any entire function $f$ goes back to the dawn of holomorphic dynamics to Fatou (\cite{Fa}) and later to Baker (\cite{BakPer}).
 A stronger conjecture was whether  for any entire transcendental function there are \emph{repelling} periodic points of any given period (except possibly period $1$) and whether there are always infinitely many such repelling periodic points for any minimal period. One of the reasons why Baker was interested in the problem of existence of repelling periodic points for entire functions is because he could show that if $f$ is entire with at least one repelling periodic point, then the centralizer of $f$ is countable (see \cite{BakPerm}). For an extensive bibliography on the subject one can consult \cite{Ber1} (See also \cite{Ber2}).


The most general result available is the   following theorem shown in  {\cite[Theorem 1]{Ber1}}: 
\begin{thm}\label{Berg}
Let $f$ be an entire transcendental function and $n\geq2$. Then $f$ has infinitely many repelling periodic points of minimal period $n$.
\end{thm}

 The proof in \cite{Ber1} uses results from Wiman-Valiron theory and a version of  Ahlfors Three  Islands Theorem due to Hayman.  Observe that the function $e^z+z$ has no fixed points, so that Theorem \ref{Berg} is optimal in this sense.

By restricting the class of entire transcendental  functions under consideration to entire transcendental functions with a  bounded set of singular values, we can show that there must also be infinitely many repelling  points of period $n=1$ and we do so with a considerably more elementary proof, using only logarithmic coordinates as in \cite{EL} and some considerations on the hyperbolic metric. Our proof also works for a wide  class of meromorphic functions. This proof although simple is new to  the author's knowledge, and has the  advantage of proving the existence of both repelling fixed points and repelling periodic points at once. Our proof also implicitly shows that there is exactly one repelling periodic point for each (allowable) periodic itinerary with respect to a specific  dynamically meaningful partition.

\begin{mainthm}
Let $f$ be an entire transcendental function or  a  meromorphic function with a bounded set  of singular values. If $f$ is meromorphic but not entire assume also that infinity is a logarithmic singularity. Then $f$ has infinitely many repelling periodic points of any given minimal period $n\geq 1$. 
\end{mainthm}
 
 Observe that the class of functions under consideration is not invariant under composition, since if $f$ is meromorphic there could be a sequence of singular values accumulating on a pole $P$ whose image is hence  unbounded. For an entire function $f$ on the other hand, if $S(f)$ is bounded also $S(f^n)$ is bounded for any $n$. For more on  the classification of singularities and a precise definition of logarithmic singularity see \cite{BE}.

The existence of repelling \emph{fixed} points for entire transcendental functions with finitely many singular values can be found in \cite{EL1}:

\begin{prop}[\cite{EL1}]\label{EL fix} Let $f$ be an entire transcendental function with finitely many singular values. Then $f$ has infinitely many repelling fixed points.
\end{prop}

In fact, the proof of Proposition~\ref{EL fix}  in \cite{EL1} (which is different from ours) can be applied also to show the existence of infinitely many repelling fixed points for  entire transcendental functions with a bounded set of singular values, while   an additional argument is needed to say that there are infinitely many repelling periodic points for any minimal period (\cite{Er}). The proof also   generalizes to meromorphic  functions belonging to  the class considered in the Main Theorem, but only proves the existence of fixed points and does not show  that the points in question are repelling.

Finally, the proof of existence of infinitely many repelling \emph{fixed} points for entire transcendental functions and  meromorphic functions with order different from zero (always with a bounded set of singular values) can be found in \cite[Theorem 2, Theorem 3]{LZ}, with some  additional information on the multipliers of these points. Their proof is quite different from the one presented in this paper. In \cite{Zh} one can find additional results for periodic points for some classes of  transcendental functions,  phrased in the language of Nevanlinna theory.

\subsection*{Acknowledgments}
This work was inspired by the proof in \cite{De}. 
The author is thankful to Carsten Petersen,  Filippo Bracci, N\'uria Fagella and Pavel Gumenyuk for useful discussions on this topic. The author would also like to thank Alexandre Eremenko and  Jian-Hua Zheng for pointing out the references \cite{EL1}, \cite{LZ}, and \cite{Zh}. I thank the referee for his/her comments and for the celerity in revising this paper. This work was  supported by the ERC grant HEVO - Holomorphic Evolution Equations n. 277691. Part of this work was carried out during the conference 'Ergodic Theory and Complex Dynamics' at the Erwin Schrodinger Institute in Vienna.

\section{Proof of the Main Theorem}
The proof is based on the tool of logarithmic coordinates introduced by Eremenko and Lyubich in \cite{EL}, slightly modified for meromorphic functions. 
The next lemma is a  basic fact in algebraic topology, see for example \cite{Ha} for the general theory about coverings.  Let $\Hyp$ denote the right half plane $\Hyp:=\{z\in\C;\;\Re z>0\}$.

\begin{lem}[Coverings of $\D^*$]\label{Coverings}
 Let $U\subset\hat{\C}$, $\D$ be the unit disk and $\D^*=\D\setminus\{0\}$. If $f$ is a holomorphic  covering from $U\ra \D^*$, then
 \begin{itemize}
\item[(a)]  either $U$ is biholomorphic to $\D^*$  and $f$ is equivalent to $z^d$- that is, there exists a biholomorphic map $\phi: U\ra \D^*$ such that $f=z^d\circ\phi$;
\item[(b)]  or $U$ is simply connected and $f$  is a universal covering, hence equivalent to the exponential map - that is, there exists a biholomorphic map $\phi: U\ra\Hyp$ such that $f=\exp\circ\phi$.
 \end{itemize}
\end{lem}

If $S(f)$ is bounded there exists a disk $D$ such that $\Omega:=\C\setminus \ov{D}$ contains no singular values. In particular for any  connected component $U $ of $f^{-1}(\Omega)$, $f$ is an (unbranched) covering from $U$ to $\Omega$. Since $\Omega\sim \D^*$, by Lemma~\ref{Coverings} for any  connected component $U$ of $f^{-1}(\Omega)$  either $U$ is  bounded and $f:U\ra\Omega$ is equivalent to  $z^d$, or $U$ is unbounded, simply connected, and $f:U\ra\Omega$ is equivalent to  $e^z$. In the first case, $U$ contains exactly one  pole $P$, and $f:U\setminus\{P\}\ra \Omega$ is a covering of degree $d$ where $d$ is the order of the pole. These bounded  components are of no interest to us.

 In the second case $U$ is called a \emph{tract}. If infinity is a logarithmic singularity there is always at least one such tract. If $f$ is entire transcendental, all connected components of the preimage of $\Omega$ are tracts. Tracts are simply connected unbounded sets. Moreover, tracts have disjoint closures and accumulate only at infinity; that is, if $z_n$ is a sequence of points all belonging to different tracts, then $z_n\ra\infty$ (see \cite{EL}).    Let $\TT$ denote the union of all tracts. The following lemma can be found in \cite[Lemma 2.1]{BF}

\begin{lem} There exists a simple (analytic) curve $\delta\subset \Omega\setminus \ov{\TT}$ connecting $\ov{D}$ to infinity.
\end{lem}

The preimages of $\delta$ partition each tract  into countably many \emph{fundamental domains} $F_i$ such that $f: F_i\ra\Omega\setminus\delta$ is univalent for any $i$.
The $F_i $ also do not accumulate on any compact set so only finitely many of them intersect $\ov{D}$.

Each tract $T$, each fundamental domain  $F$,  and the set  $\Omega\setminus\delta$ are all simply connected open sets admitting a hyperbolic metric. For any set $U$ admitting a hyperbolic metric we denote its density by   $\lambda_U$, and the hyperbolic distance in $U$ by $d_U$.

We now show that in a neighborhood of infinity the hyperbolic density of any fundamental domain $F$ is much larger than the hyperbolic density of the tract that contains it, and hence it is much larger than the hyperbolic density of $\Omega\setminus\delta$.

\begin{lem}\label{Local} For any fundamental domain $F$ there exists $\kappa_F<1$ and a neighborhood $U_F$ of infinity in $\Chat$ such that 

\begin{equation}
\frac{\lambda_{\Omega\setminus\delta}(z)}{\lambda_{F}(z)}\leq \kappa_F<1 \ \forall z\in F\cap U_F.
\end{equation}
In fact, by restricting $U_F$, $\kappa_F$ can be taken to be arbitrarily small.
\end{lem}

\begin{proof}
Let $T$ be the tract containing $F$. Since $f: T\rightarrow \Omega\sim \D^*$ is a universal covering, there exists a conformal isomorphism $\phi: T\rightarrow \Hyp$ where $\Hyp:=\{z\in\C: \Re z>0\}$  and such that $f|_T=\exp\circ\phi$ (see Lemma~\ref{Coverings} part (b) and Figure~\ref{LocalFig}). Using this equation,  the fact that the exponential is $2\pi i $ periodic and the fact that the $F_i$ are fundamental  domains for $f$ we obtain the following: For any fundamental domain $\tilde{F}\neq F$ contained in $T$ we have that $\phi(\tilde{F})=\phi(F)+2\pi i k$ for some $k\in\Z$. In particular $\phi(F)$ contains no vertical segments of diameter bigger than $2\pi$, so for any $z\in \phi(F)$, the Euclidean distance $\dist(z,\partial F)$ is less than  or equal to $\pi$.  Since $\phi(F)$ is simply connected, by standard estimates on the hyperbolic metric (see e.g. \cite[Theorem 8.6]{BM}) we have that

\[\lambda_{\phi(F)}(z)\geq\frac{1}{2\pi}\ \ \forall z\in \phi(F).\]

Since $\phi$ is a conformal  isomorphism it is an isometry for the hyperbolic metric (see e.g. \cite[Theorem 6.3]{BM}), hence using the fact that $\lambda_{\Hyp}(z)=\frac{1}{2\Re z}$ we have that

\begin{equation} \label{Star}
\frac{\lambda_T(z)}{\lambda_F(z)}=\frac{\lambda_{\phi(T)}(\phi(z))}{\lambda_{\phi(F)}(\phi(z))}=\frac{\lambda_{\Hyp}(\phi(z))}{\lambda_{\phi(F)}(\phi(z))}\leq \frac{\pi}{\Re\phi(z)} \ \ \forall z\in F.
\end{equation}

In particular $\frac{\lambda_T(z)}{\lambda_F(z)}$ can be made arbitrarily small by letting $\Re \phi(z)$ be arbitrarily large, that is the same as saying, $|z|$ be arbitrarily large.

Now choose $\kappa_F<1$,  let $U_F\subset\Chat$ be a neighborhood of infinity small enough such that $\frac{\lambda_{\Omega\setminus\delta}(z)}{ \lambda_T(z)} \leq 1$in $U_F$ (this can be done using the comparison principle for the hyperbolic metric because $T\subset \Omega\setminus\delta$ in a neighborhood of infinity) and such that $\Re \phi(z)\geq\frac{\pi}{ \kappa_F}$ for $U_F\cap F$ (this can be done because $\Re\phi(z)\rightarrow\infty$ as $|z|\rightarrow\infty$).

It follows using (\ref{Star}) that

\[\frac{\lambda_{\Omega\setminus\delta}(z)}{\lambda_F(z)}\leq \frac{\lambda_T(z)}{\lambda_F(z)}\leq \frac{\pi}{\Re\phi(z)} \leq\kappa_F \ \ \forall z\in F\cap U_F\]

as required. 

\end{proof}

\begin{cor}\label{strictlemma}
For any fundamental domain $F$ not intersecting $D$ there exists $\kappa_F<1$  such that 

\begin{equation}\label{strict}
\frac{\lambda_{\Omega\setminus\delta}(z)}{\lambda_{F}(z)}\leq \kappa_F<1 \ \forall z\in F.
\end{equation}
\end{cor}
 \begin{proof}
 Since $F$ is compactly contained in $\Omega\setminus\delta$ except in a neighborhood of infinity, the inequality follows from the comparison principle except in a neighborhood of infinity, where it follows from Lemma~\ref{Local}. Observe that in  this case, in order for (\ref{strict}) to hold for all $z\in F$, $\kappa_F$  can no longer be chosen to be arbitrarily small.
 \end{proof}

 \begin{figure}[hbt!]
\begin{center}
\def\svgwidth{13cm}
\begingroup%
  \makeatletter%
  \providecommand\color[2][]{%
    \renewcommand\color[2][]{}%
  }%
  \providecommand\transparent[1]{%
    \errmessage{(Inkscape) Transparency is used (non-zero) for the text in Inkscape, but the package 'transparent.sty' is not loaded}%
    \renewcommand\transparent[1]{}%
  }%
  \providecommand\rotatebox[2]{#2}%
  \ifx\svgwidth\undefined%
    \setlength{\unitlength}{1187.88217773bp}%
    \ifx\svgscale\undefined%
      \relax%
    \else%
      \setlength{\unitlength}{\unitlength * \real{\svgscale}}%
    \fi%
  \else%
    \setlength{\unitlength}{\svgwidth}%
  \fi%
  \global\let\svgwidth\undefined%
  \global\let\svgscale\undefined%
  \makeatother%
  \begin{picture}(1,0.78372098)%
    \put(0,0){\includegraphics[width=\unitlength]{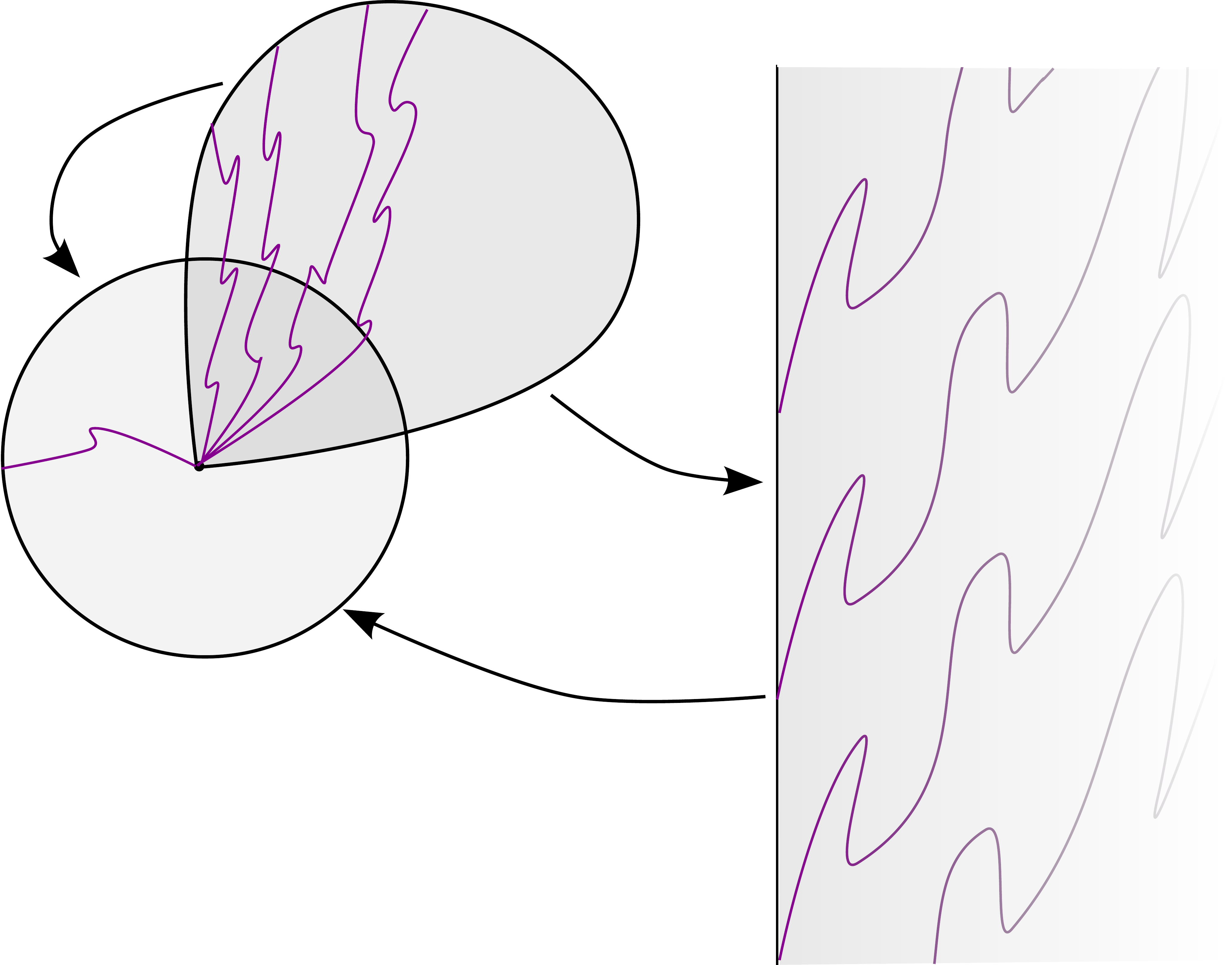}}%
    \put(0.16523329,0.37033587){\color[rgb]{0,0,0}\makebox(0,0)[lb]{\smash{$\infty$}}}%
    \put(-0.0007611,0.24243854){\color[rgb]{0,0,0}\makebox(0,0)[lb]{\smash{$\Omega\setminus\delta$}}}%
    \put(0.03869659,0.38666318){\color[rgb]{0,0,0}\makebox(0,0)[lb]{\smash{$\delta$}}}%
    \put(0.51763122,0.43972696){\color[rgb]{0,0,0}\makebox(0,0)[lb]{\smash{$\phi$}}}%
    \put(0.40606123,0.18257171){\color[rgb]{0,0,0}\makebox(0,0)[lb]{\smash{$\exp$}}}%
    \put(0.012845,0.66014573){\color[rgb]{0,0,0}\makebox(0,0)[lb]{\smash{$f$}}}%
    \put(0.52715549,0.60844256){\color[rgb]{0,0,0}\makebox(0,0)[lb]{\smash{$T$}}}%
    \put(0.24823047,0.67103061){\color[rgb]{0,0,0}\makebox(0,0)[lb]{\smash{$F$}}}%
    \put(0.78,0.40435107){\color[rgb]{0,0,0}\makebox(0,0)[lb]{\smash{$\phi(F)$}}}%
     \put(0.6,0.3){\color[rgb]{0,0,0}\makebox(0,0)[lb]{\smash{$\Hyp$}}}%
  \end{picture}%
\endgroup%
\end{center}
\caption{\small The construction used in the proof of Lemma~\ref{Local}. For simplicity only one tract and three of fundamental domains within  are drawn, and the possible  spiralling near infinity is not noticeable on the left-hand side.} 
\label{LocalFig}
\end{figure}
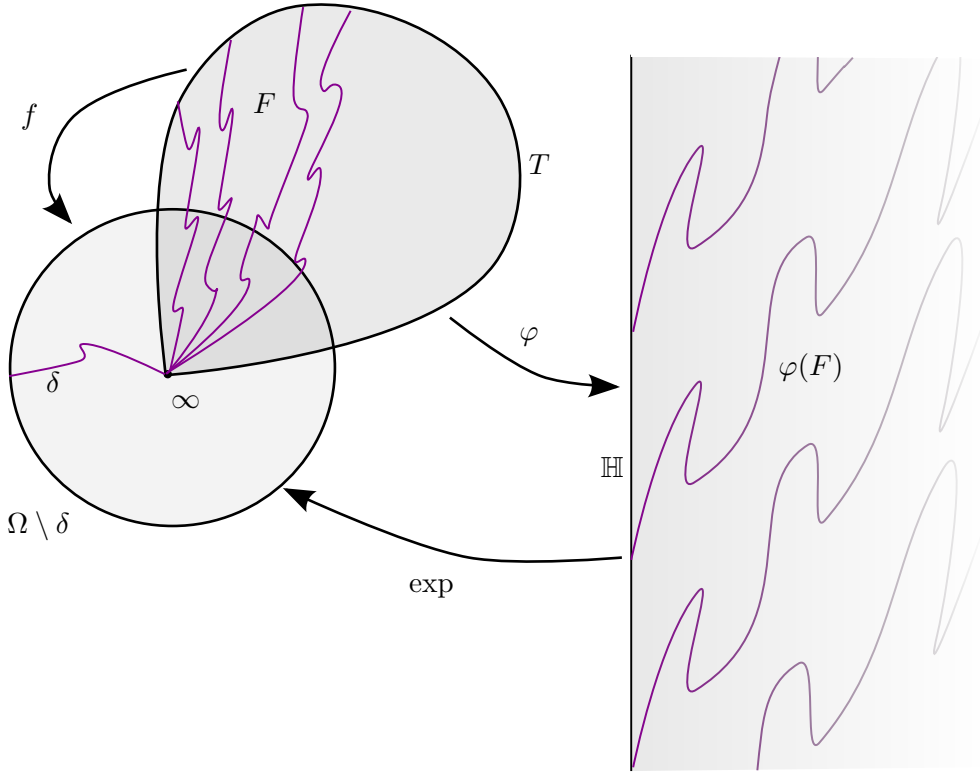

\begin{proof}[Proof of the Main Theorem]

 Let us first show that $f$ has infinitely many repelling fixed points. Since there are infinitely many fundamental domains only finitely many of which intersect the disk, it is enough to show that for each  non-intersecting fundamental domain $F$ there exists a repelling fixed point $w_F$ in ${F}$.
Since $f:F\ra \Omega\setminus\delta$ is univalent and surjective, there exists a unique univalent inverse branch $\psi_F: \Omega\setminus \delta\ra F$. In view of the Banach Fixed point Theorem it is enough to show that $\psi_F$ strictly contracts the hyperbolic metric in $F$ and to show that the fixed point $w_F$ given by the Banach's Theorem is not on the boundary of $F$. By Corollary~\ref{strictlemma} there exists $\kappa_F<1$ such that $\lambda_{\Omega(z)\setminus\delta}\leq\kappa_F\lambda_F(z)$. 
Since $\psi_F$ is univalent hence an isometry between the hyperbolic metric of $\Omega\setminus\delta$ and   the hyperbolic metric of $F$, we have that for any $w,z\in F$ 

\[d_F(\psi_F(w),\psi_F(z))=d_{\Omega\setminus\delta}(w,z)\leq\kappa_F\;  d_F (w,z)\]
By the Banach Fixed point Theorem there exists $w_F\in\ov{F}$ such that $\psi_F^n$ converges to $w_F$ uniformly on  compact subsets of $F$. Since $\psi_F$ strictly contracts the hyperbolic metric, for any $x\in F$ the points $\{\psi_F^n(x)\}_{n\in\N}$ form a Cauchy sequence and the distance $ d_F(x,w_F)$ is finite. Since by completeness of the hyperbolic metric for any $x\in F$ the hyperbolic  distance to the boundary is infinite, we have that  $w_F$ is not on the boundary hence that it belongs to $F$. Since $w_F$ is an attracting fixed point for an inverse of $f$, it is a repelling fixed point for $f$.
 
Let us now fix $n\in \N$, and consider any sequence $s$ of fundamental domains $s=F_1\ldots F_n$ such that $F_i\cap \ov{D}=\emptyset$ for all $i=1\ldots n$. It is clear that there are infinitely many different such choices for any $n$. 
 For any $i=1\ldots n$ let $\psi_{F_i}$ be the univalent inverse branch from $\Omega$ to $F_i$. Define
 
 \[\psi_s:=\psi_{F_1}\circ\psi_{F_{2}}\circ\ldots\circ\psi_{F_n}.\]
 Since none of the $F_i$ intersects $\ov{D}$, $\psi_s:\Omega\ra F_1$ is univalent and well defined and it is a strict contraction by Corollary~\ref{strictlemma}, hence by the previous argument  it has a attracting fixed point $w_s\in F_1$, which is by definition a repelling fixed point  for $f^n$. By choosing for example  $F_i\neq F_j$ for $i\neq j$, since  by construction $f^{i-1}(z_s)\in F_{i}$ for $i=1\ldots n$, we can ensure  that $w_s$ has minimal period $n$. Since the initial fundamental domain $F_1$ can be chosen in infinitely many different ways, there are infinitely many such periodic points  of minimal period $n$. In fact, more can be said: if $s,\ell$ are two different sequences of fundamental domains of the same length $n$, and $z_s,z_\ell$ are the corresponding repelling periodic  points constructed for $\psi_s$ and $\psi_\ell$,     since $s\neq \ell$ we have that $f^i(z_s)\neq f^i(z_\ell)$ for some $i\leq n$, hence $z_s\neq z_\ell$. In particular  any  two different sequences $s$ and $\ell$ of fundamental domains give rise to different repelling periodic points for $f$. 
\end{proof}

%
%
\begin{rem} While proving the Main Theorem we constructed repelling periodic points with any prescribed periodic itinerary with respect  to the fundamental domains not intersecting the disk $\ov{\D}$, and we  showed that for any given itinerary such a periodic point is unique. There may certainly be other periodic points, corresponding to fundamental domains which do intersect $\ov{\D}$ (see \cite{BF}). It is also possible to show that there are also repelling periodic points associated to poles far enough from $\ov{D}$ (see \cite{BF2}), and to classify the remaining periodic points for $f$. The  idea of constructing points in the Julia set using their symbolic itineraries is reminiscent of techniques used by several authors in complex dynamics, see for example  (\cite{DT}) and (\cite{BK}).
\end{rem}

\end{document}